\documentclass[11pt]{amsart}

\usepackage{enumerate}
\usepackage{graphicx,graphics}
\usepackage{amsfonts}
\usepackage{amssymb}
\usepackage{amsthm}
\usepackage{amsmath}
\usepackage{mathrsfs}
\usepackage[all]{xy}
\usepackage{tcolorbox}
\usepackage[english]{babel}
\usepackage{xcolor}
\input{xy}
\xyoption{all}

\DeclareMathOperator{\Lip}{Lip}
\DeclareMathOperator{\LLip}{LLip}
\DeclareMathOperator{\LcLip}{L_cLip}

\def\R{\mathbb{R}}
\def\C{C}

\newtheorem{theorem}{Theorem}[section]
\newtheorem{lemma}[theorem]{Lemma}

\newtheorem{remark}[theorem]{Remark}
\newtheorem{corollary}[theorem]{Corollary}

\newtheorem{proposition}[theorem]{Proposition}
\newtheorem{example}[theorem]{Example}
\newtheorem{definition}[theorem]{Definition}

\begin{document}

\title{Lattice Lipschitz operators on $C(K)-$space}

\author[R. Arnau, J. M. Calabuig and E.\ A.\ S\'{a}nchez-P\'{e}rez]{Roger  Arnau, Jose M. Calabuig  and     Enrique \ A.\ S\'{a}nchez-P\'{e}rez$^*$}


\address{   Instituto Universitario de Matem\'atica Pura y Aplicada,
              Universitat Polit\`ecnica de Val\`encia \\ Camino de Vera s/n, 46022
              Valencia, Spain. }
            \email{ararnnot@posgrado.upv.es, jmcalabu@mat.upv.es, easancpe@mat.upv.es}

\maketitle

\begin{abstract}
Given a Banach lattice $L,$ the space of lattice Lipschitz operators on $L$ has been introduced as a natural  Lipschitz generalization of the linear notions of diagonal operator and multiplication operator on Banach function lattices. It is a particular space of superposition operators on Banach lattices. Motivated by certain procedures in Reinforcement Learning based on McShane-Whitney extensions of Lipschitz maps, this class has proven to be useful also in the classical context of Mathematical Analysis. In this paper we discuss the properties of such operators when defined on spaces of continuous functions, focusing attention on the functional bounds for the pointwise Lipschitz inequalities defining the lattice Lipschitz operators, the representation theorems for these operators as vector-valued functions and the corresponding dual spaces. Finally, and with possible applications in Artificial Intelligence in mind, we provide a McShane-Whitney extension theorem for these operators.
\end{abstract}

\vspace{0.3cm}

\noindent
\keywords{Lipschitz function, continuous function, $C(K)$,  lattice Lipschitz, non-linear operator, Banach lattice.}

\vspace{0.3cm}

\noindent
\subjclass{AMS: 26A16,   
 46E15, 
47B38  	
 }

\vspace{0.3cm}
\noindent
\thanks{{*}Corresponding author. The first author was supported by a contract of the Programa de Ayudas 
de Investigaci\'on y Desarrollo (PAID-01-21), Universitat 
Polit\`ecnica de Val\`encia. This publication is part of the R\&D\&I project/grant 
PID2022-138342NB-I00 funded by MCIN/ AEI/10.13039/501100011033/ (Spain). }

\section{Introduction}

%

Lattice Lipschitz operators have been introduced as a natural extension of the notion of diagonal operator on Euclidean lattices, as efficient tools for the representation, approximation and extension of Lipschitz maps on these spaces (\cite{arn1, arn2}). Although the first study has been performed on finite-dimensional \cite{arn1} spaces, it also makes sense in the context of classical lattices of  functions such as $C(K)$ or Banach function spaces, extending in this case the notion of multiplication operator on function spaces to the setting of Lipschitz functions. 

The original motivation for the study of this class was the goal of providing new functional tools for Artificial Intelligence to generalize the role played by real-valued Lipschitz functions and their extensions in this field, a hot topic of current research. Indeed,
Lipschitz functions are a usual tool in machine learning that allows the interchange of ideas from pure mathematical frameworks to applied ones \cite{ar,lel}. In recent times, it has become relevant for example in deep learning and reinforcement learning.
Control on neural networks approximation is one of the most common context of application of Lipschitz functions, in the context of the so-called Lipschitz neural networks \cite{le,ta}. However, Lipschitz functions appear in different contexts, as function approximation \cite{ani,lel} or reinforcement learning with reward functions with Lipschitz control \cite{as,ca,ky,li,vo}.

However, it seems to be also a useful tool for the analysis of certain issues of Functional Analysis involving Lipschitz operators on Banach spaces (summability, factorization, nonlinear geometry,...) which are also central topics of current research in this area of mathematics (see \cite{ac,an,ch,fa,fe,ma,ru,ya} and the references therein).  This is our motivation for the present paper. We show that lattice Lipschitz operators are diagonal Lipschitz maps emulating the linear case, in which multiplication operators are the canonical diagonal maps. It must be said that multiplication operators (operators on function spaces defined by pointwise multiplication by a given function) are necessary tools for some classical fields of the Functional Analysis, such as the Maurey-Rosenthal factorization of $p-$concave maps \cite{defa}, and the Pisier factorization of $(q,1)-$summing operators acting on $C(K)-$spaces \cite{pisier}. Also in summability theory, more related in this case to the Pietsch-type factorization of $p-$summing operators and its many variants \cite{dies}. As proposed in \cite{ArnpreBfs} for the case of Banach function spaces, we intend to introduce lattice Lipschitz operators on $C(K)-$spaces to provide the equivalent tool for new developments on summability of Lipschitz maps.

From the theoretical point of view, some general facts about the family of functions we are interested in are known since the last quarter of the twentieth century. In fact, some of their properties have been studied in depth by J. Appell and  some of his collaborators  in the context of what is known in the mathematical literature as superposition operators \cite{appel2,appel3,appLib}.  In fact, some relevant characterizations of locally Lipschitz superposition operators satisfying inequalities like the ones we use in our definitions can be found in some classical works, and some of the main results we consider can be found there (see e.g. \cite[Th. 1]{appell}, \cite[Th.1]{appel3}).  In this paper we face the problem of describing such a family of maps by means of an identification procedure that relates operators to vector-valued functions and elements of topological tensor products, thus completing what is known about such maps.

A relevant property that makes lattice Lipschitz operators interesting for Machine Learning and also for Functional Analysis is that, under some restrictions on the domain space (Euclidean spaces, Hilbert spaces, Banach function spaces, among others), operators of this class satisfy McShane-Whitney type extension theorems \cite{ar}. As we have explained above, this have proved to be fruitful for applications, allowing a general functional form for Lipschitz regression in the same way that these classical extension theorems are used in the real-valued case. 

In this paper we investigate the main properties of lattice Lipschitz operators on $C(K)-$spaces. As we already said, the main idea is that lattice Lipschitz operators are diagonal, mimicking the role of the multiplication operators on $C(K)-$spaces defined by functions belonging to $C(K)$ in the Lipschitz functional context.   In Section \ref{S2} we discuss the properties of the bound functions that appear in the inequalities of the lattice Lipschitz operators, providing equivalent characterizations and weakening the requirements in the definition of these inequalities. We show that  we can always find a continuous function as bound, while the minimal function satisfying the inequalities is not in general continuous. In Section \ref{S3} we give several representation results for lattice Lipschitz operators as vector-valued functions, what provides in a natural way some special subspaces of lattice Lipschitz maps in which the continuity properties are better.  
As a consequence of the representation theorems we get, we show also structure results.  For example, we show that as in the case of multiplication operators on $C(K)-$spaces, due to the diagonal nature of the maps involved, the space of lattice Lipschitz operators is a (non-commutative) unital algebra with the usual composition of operators. 
This, together with Section \ref{S2}, gives the general picture of what lattice Lipschitz operators are, and allows to find some concrete applications, that are shown in Section \ref{S4}. The first one is the identification of the duals of some spaces of lattice Lipschitz maps. The second one are bilinear and operator-valued representations of the dual space of the main class of lattice Lipschitz operators. 
Finally, the last subsection is devoted to the extension properties of these maps, which is possibly the most productive application taking into account the suitable use of these results in Lipschitz regression and Machine Learning. We show that, given a lattice Lipschitz operator on  any subset of $C(K),$ we can always extend it to the whole $C(K)$ preserving any continuous functional bound for the operator. Some applications are also shown.

Before starting to explain our results, let us mention some ideas  (mainly due to Kalton and Godefroy) on Lipschitz maps defined on $C(K)-$spaces \cite{go1,go2,ka0,Kal03,ka1,ka2}. Although there are obvious connections with our ideas, it must be said that we have not used these results, due to the fact that lattice Lipschitz operators are a rather particular subspace of the usual Lipschitz maps. However, note that the main topic that is studied in these papers is the same that we analyze here: the possibility of extending a map acting in a subspace to a well-defined Lipschitz operator on $C(K).$

We will use standard Banach lattices notation and results. If $X$ is a Banach space, we write $L(X)$ as the space of linear operators from $X$ to $X.$ If $K$ is a compact Hausdorff space, we write $C(K)$ for the space of continuous real functions on $K$ with the usual supremum norm. We will write $B(K)$ for the space of (non-necessarily continuous) bounded  real functions  acting in $K.$
For the aim of simplicity, we will often use the notation of the order in the lattice ($f \le g$) instead of doing an explicit reference of the evaluation at any point $w$ 
($f(w) \le g(w)$).  We will write $Lip(\mathbb R)$ for the space of real valued Lipschitz function of real variable, and $Lip(\phi)$ for the Lipschitz norm of such a function $\phi \in Lip(\mathbb R).$
The reader can find definitions and  fundamental results on Lipschitz functions in \cite{cobzas}.


%

\vspace{0.5cm}


\section{Uniform boundedness and optimal bound functions for lattice Lipschitz operators on $C(K).$}  \label{S2}

In this section we present the definition and main properties of the lattice Lipschitz operators defined on $C(K)-$spaces.  The direct adaptation of the original definition  gives the following.

\begin{definition} (adapted from Definition 1 in \cite{arn1})
Let $B \subseteq C(K).$
A map $T: B \to C(K)$ is a  lattice Lipschitz operator if there is a function $\varphi \in C(K)$ such that for every two functions $f,g \in B$ we have the inequality
\begin{equation}  
	\label{eq:llip_inequality}
	\big| T(f)(w)-T(g)(w) \big| \le \varphi(w) \, \big| f(w)-g(w) \big|.
\end{equation}
We will write $\LLip(B,\C(K))$ for the space of all the lattice Lipschitz operators from $B$ to $\C(K)$, and $\LLip(\C(K))$ when $B = \C(K)$.
\end{definition}

In the context of superposition operators, the inequalities constituting the above definition are in fact, under certain requirements, equivalent to being a Lipschitz map, as can be seen in Theorem 1 of \cite{appell} for the case $K=[0,1]$ (see also Th.6.6 in \cite[\S6]{appLib}). Also, Theorem 1 of \cite{appel3} shows that, again under certain restriction, the constant appearing in the inequality can be defined locally and thus a representation as  function of $C(K)$ of the ratio among both parts of the inequality  (\ref{eq:llip_inequality}) can be provided. 

\medskip
In particular, one can easily see by simply taking supremum on both sides of the above inequality that any lattice Lipschitz operator is Lipschitz.
 Let us show some examples. 

\begin{example} \label{exa1}
	The first one is in a sense standard. It is defined by using the composition with a real-valued Lipschitz function of real variable.
	\begin{enumerate}[i)]
		\item Let $k  >0$ and consider the function $T_k(f)(w):= \frac{k}{k+ |f(w)|};$ since
		\begin{align*}
			\big| T(f)(w) - T(g)(w) \big| & = \left| \frac{k}{k + |f(w)|} -  \frac{k}{k+ |g(w)|} \right| =   \frac{ k \, \big| | f(w)| - |g(w)| \big| }{ \big( k + |f(w)| \big) \big( k + |g(w)| \big) } \\
			& \le \frac{ k \, \big| f(w) - g(w) \big| }{ \big( k + |f(w)| \big) \big( k + |g(w)| \big) } \le  \frac{1}{k} \, \big|  f(w) - g(w) \big|
		\end{align*}
		we have that $T_k$ is lattice Lipschitz with bound function $\varphi$ less or equal to (the constant function) $1/k.$
		
		\item However, in a sense, the lattice Lipschitz operators are a generalization of the notion of (linear) multiplication operator on $C(K)-$spaces, that are well-defined as a consequence of the Banach algebra structure of the spaces of continuous functions; ``diagonality" is  the main feature of our class of Lipschitz-type maps.
		Indeed, let $ h \in C(K)$ and consider the multiplication operator $M_h: C(K) \to C(K)$ given by $M_h(f)= h \cdot f,$ $f \in C(K).$ Then $M_h$ is lattice Lipschitz, as a consequence of the inequalities
		$$
		|M_h(f)(w) - M_h(g)(w) | = | h(w) f(w) - h(w) g(w) | \le  | h(w) | \, | f(w) - g(w) |
		$$
		for all $f,g \in C(K),$
		where the bound function is $\varphi= |h|.$
	\end{enumerate}
\end{example}

\medskip

Due to the properties of $C(K),$ the following characterization of the lattice Lipschitz operators is trivially satisfied. Note, however, this does not hold in general  for lattice Lipschitz operators on other function spaces (see S. 2 the examples in S.3 of  \cite{ArnpreBfs}).

\begin{remark} \label{remconstant}
	Let $B \subseteq C(K).$
	A map
	$T: B \to C(K)$ is  a  lattice Lipschitz operator if  and only if there is a constant $Q >0$ such that for every two functions $f,g \in B$ we have the inequality
	\begin{equation}
		\label{eq:Q_uniform_bound}
		\big| T(f)-T(g) \big| \le Q \, \big| f-g \big|.
	\end{equation}
	To see this, it is enough to consider the constant function $Q \chi_K$ for $Q = \|\varphi\|_{C(K)},$ that belongs to $C(K)$ and satisfies
	$\varphi \le Q \chi_K,$ and so the inequality holds.
\end{remark}

This fact allows to define a norm on the space of lattice Lipschitz operators as follows.

\begin{definition}
	Given $T \in \LLip(B,\C(K))$, we define $ \| T \|_{\LLip}$ as the infimum of all $Q > 0$ satisfying \eqref{eq:Q_uniform_bound}, that is,
	\begin{equation}
		\label{eq:llip_norm}
		\| T \|_{\LLip} = \inf \big\{ \| \varphi \|_\infty : \varphi \in \C(K) \ \mathrm{satisfies} \ \eqref{eq:llip_inequality} \}.
	\end{equation} 
\end{definition}
It is clearly a norm. Let us show that the corresponding space is in fact a Banach space.
\begin{proposition}
	\label{propo:llip_banach}
	The space of lattice Lipschitz operators on $\C(K)$ that vanish at $0,$ $\big( \LLip_0(\C(K)), \| \cdot \|_{\LLip} \big)$ is a Banach space.
\end{proposition}

\begin{proof}
	Standard computations show that it is a normed space.
	Consider now a Cauchy sequence $T_n \in \LLip(\C(K))$. For any $f \in \C(K)$ and $w \in K$, $ \{ T_n(f)(w) \} $ is a Cauchy sequence in $\R$, so we can define $T(f)(w)$ as its limit. Moreover, $T_n(f) \to T(f)$ uniformly on $w \in K$, so $T: \C(K) \to \C(K)$ is well defined.
	To show that $T \in \LLip(C(K))$ and $\| T_n - T \|_{\LLip} \to 0$, fix $\varepsilon > 0$, so there exists $n_0 \in \mathbb N$ such that $\| T_n - T_m \|_{\LLip} < \varepsilon$ for $n, m \geq n_0$.
	Let now $f, g \in \C(K)$ and $w \in K$. Then for any $n \geq n_0$,
	\begin{align*}
		\big| (T_n-T)&(f)(w) - (T_n-T)(g)(w) \big| \\
		& = \lim_{m \to \infty} \big| T_n(f)(w) - T_m(f)(w) - T_n(g)(w) + T_m(g)(w) \big| \\
		& \leq \varepsilon \cdot \big| f(w) - g(w) \big|.
	\end{align*}
\end{proof}

The norm defined above for lattice Lispchitz operators coincides with the Lipschitz norm of the map from $\C(K)$ to $\C(K)$, where the usual (supremum) norm is considered on the space of continuous functions. Therefore, we can consider $(\LLip_0(C(K)), \| \cdot \|_{\LLip})$ as a closed subspace of $(\Lip(\C(K),\C(K)), \| \cdot \|_{\Lip})$.

\begin{proposition}
	For any $T \in \LLip(\C(K))$, we have $\| T \|_{\LLip} = \| T \|_{\Lip}$.
\end{proposition}

\begin{proof}
	Let $T : \C(K) \to \C(K)$ be a non-null lattice Lipschitz operator. It is obvious that $T$ is a Lipschitz map.
	Now, let $\varepsilon > 0$. Then there exists $f, g \in \C(K)$ such that
	\begin{align*}
		\| T \|_{\Lip} - \varepsilon & < \dfrac{ \| T(f) - T(g) \|_\infty }{ \| f - g \|_\infty }
		= \dfrac{ | T(f)(w) - T(g)(w) | }{ | f(k) - g(k) | } \\
		& \leq \dfrac{ | T(f)(w) - T(g)(w) | }{ | f(w) - g(w) | }
		\leq \| T \|_{\LLip},
	\end{align*}
	where $w, k \in K$ are the points at which the respective suprema are attained. Letting $\varepsilon \to 0$, $\| T \|_{\Lip} \leq \| T \|_{\LLip}$.
	For the reverse inequality, given $\varepsilon > 0$ we have some $f, g \in \C(K)$ (such we can assume constant on $K$) and $w \in K$ satisfying
	\begin{align*}
		\| T \|_{\LLip} - \varepsilon & < \dfrac{ T(f)(w) - T(g)(w) }{ f(w) - g(w) }
		= \dfrac{ T(f)(w) - T(g)(w) }{ \| f - g \|_\infty } \\
		& \leq \dfrac{ \| T(f)(w) - T(g)(w) \|_\infty }{ \| f - g \|_\infty }
		\leq \| T \|_{\Lip}.
	\end{align*}
\end{proof}


A similar result can be found in \cite{appell} (Theorem 1). This is written for $K = [0,1] \subset \mathbb R,$ and from the point of view that the map considered is a superposition operator, and not starting from the pointwise inequality we use as a definition.

However, and in order to get a good description of the properties of the lattice Lipschitz operators, we are interested in finding the best function bound $\varphi$ that can be written in the inequalities that satisfy the lattice Lipschitz map.
Let us start by showing that this uniform  boundedness property that is required for all couples of functions  in the set $B$ can be reduced to the existence of such a bound for finite subsets of functions in $B,$ whenever these function bounds belong to a  weakly compact set. This automatically can be rewritten as a summability property for the operator $T,$ as will be shown in what follows.

Let us recall first some basic facts on weakly compact subsets of $C(K).$ The Krein-\v{S}mulian Theorem states that the closed convex hull of a weakly compact subset of a Banach space is  weakly compact (6.38 of \cite{ali}). A result by Grothendieck states that a subset $S \subset C(K)$ is weakly  compact if and only if it is bounded and compact  in the topology of poinwise convergence (Th.5 in \cite{gro}, see also \S4 in \cite{floret}). Next result shows that it is enough to find bound functions for $T$ on finite subsets of $B$ whenever such bound functions belong to a weakly compact convex subset $S$ of $C(K),$ and in this case  a bound function belonging to $S$ and valid for all functions can be always found.

%


\begin{lemma}  \label{finitesubs}
Let $S \subset C(K)^+$ be a weakly compact set (equivalently, by the Grothendieck's result cited above, bounded, and compact with respect to the topology of pointwise convergence), and let $B \subseteq C(K).$  Let $T: B \to C(K)$ be a map. Suppose that for every finite set $ \{f_1,g_1,...,f_n,g_n \}= W \subseteq B$ there is a function $\varphi_W \in S$ such that 
$$
|T(f_i)(w)-T(g_i)(w)| \le \varphi_W(w) \, | f_i(w)-g_i (w)|, \quad w \in K,  \quad i=1,...,n.
$$
Then there is a function $\varphi_0 \in \overline{co(S)} $ such that
$$
\big| T(f)(w)-T(g)(w) \big| \le \varphi_0(w) \, \big| f(w)-g(w) \big|, \quad w \in K,
$$
for every $f,g \in S.$
\end{lemma}
\begin{proof}
Since $S$ is weakly compact,  we also have that by the Krein-\v{S}mulian Theorem,   the (norm) closed convex hull
$\overline{co(S)}$ is weakly compact too. 

We can consider the family of all the functions $\Phi: \overline{co(S)} \to \mathbb R$ as
$$
\Phi(\varphi) := \sum_{i=1}^n \alpha_i \int_K |T(f_i)-T(g_i)| d \mu_i -  \, \sum_{i=1}^n  \alpha_i  \int_K |f_i - g_i|  \,  \varphi  \, d \mu_i,
$$
where the functions $f_1,g_1,...,f_n,g_n$ belong to $B$ and $\mu_i$ are probability measures on $K.$ Note that they are well-defined. These functions are continuous with respect to the weak topology of $C(K).$ It is clear that the set of all these functions is concave, and each of the functions is convex. Moreover, by the inequality provided by the original inequality for all couples of functions $f,g \in B,$ for each  function $\Phi$  there is an element $\varphi \in S$ such that $\Phi(\varphi) \le 0.$ A standard separation argument using Ky Fan's Lemma (see for example 9.10 in \cite{dies}) gives the result, taking into account the evaluation of the  resulting inequality for single couples of functions and $\mu_i= \delta_{w},$ the evaluation at any point $w \in K.$

\end{proof}

\begin{theorem}
Let   $B \subseteq C(K)$ and  let $T: B \to C(K)$ be a map.  Consider a weakly compact subset $S \subset C(K).$
The following statemens are equivalent.

\begin{itemize}

\item[(i)] $T$ is lattice Lipschitz with a function bound $\varphi \in \overline{co(S)}.$

\item[(ii)] For every finite set $ \{f_1,g_1,...,f_n,g_n \}= W \subseteq B$ there is a function $\varphi_W \in S$ such that  for every 
$w \in K,$
$$
|T(f_i)(w)-T(g_i)(w)| \le \varphi_W(w) \, | f_i(w)-g_i (w)|,  \quad i=1,...,n.
$$

\item[(iii)] For every weak Cauchy series $s= \sum_{i=1}^\infty |f_i-g_i|$ with $f_i ,g_i \in B$ there is a function $\varphi_{s} \in S$ such that for  every $N \subseteq \mathbb N$ 
  and each probability measure $\mu \in C(K)^*,$
$$
\sum_{i \in N}\langle |T(f_i)-T(g_i)|, \mu \rangle   \le \sum_{i  \in N}  \langle \varphi_{s} |f_i-g_i|, \mu \rangle.
$$
\end{itemize}

\end{theorem}
\begin{proof}
(i) $\Rightarrow$ (ii) is obvious, and the converse is given by Lemma \ref{finitesubs}. 

Note that  (iii) $\Rightarrow$ (ii); to see this, it is enough to consider finite sets $N$ in (iii) and $\mu=\delta_w$ for every $w \in K.$

Let  us show that (i) $\Rightarrow$ (iii). We already have a non-negative function $\varphi \in \overline{co(S)}$ satisfying the pointwise domination.  Take a weak Cauchy series $\sum_{i=1}^\infty |f_i-g_i|$ and $N \subseteq \mathbb N.$ Since for all  $i \in \mathbb N$ we have that 
$$
\varphi |f_i-g_i| \le \|\varphi\|_{C(K)} \, |f_i-g_i|,
$$
we have that  for a probability measure $\mu \in C(K)^*$ and $n \in \mathbb N,$
$$
\sum_{i \in N \cap \{1,...,n\}} \langle \varphi |f_i - g_i|, \mu \rangle \le \|\varphi\|_{C(K)}   \sum_{i \in N \cap \{1,...,n\}} \langle |f_i - g_i|, \mu \rangle
$$
$$
 \le \|\varphi\|_{C(K)}   \sum_{i =1}^\infty \langle |f_i - g_i|, \mu \rangle  < \infty,
$$
and so the series defined by the terms in the left hand side converges for every $\mu.$ The domination inequality gives also that 
$$
\sum_{i \in N} \langle |T(f_i) - T( g_i)| , \mu \rangle \le \sum_{i  \in N} \langle \varphi |f_i-g_i|, \mu \rangle
$$
and so the result holds for $\varphi_N=\varphi$ for every $N \subset \mathbb N.$

\end{proof}


Although this theoretical result provides a broader characterisation than that given by the existence of a uniform numerical bound for all pairs $f,g \in B,$ we are concerned with finding concrete realisations of suitable families $S$ of functions to operate with.  In fact, we are interested in finding the best bound function $\varphi$ that can be found for a given lattice Lipschitz operator. Let us show in what follows that such  function is not in general continuous, that is, it does not belong to $C(K).$


Let  $T:B \to C(K)$ be a lattice Lipschitz operator and for $f,g \in B$  consider the functions 
$$
\varphi_{f,g}(w) = \frac{| T(f)(w)- T(g)(w)|}{|f(w)-g(w)|},  \quad w \in K \quad \text{in the case that $f(w)\ne g(w),$}
$$
and 
$\varphi_{f,g}(w) = 0$ if $f(w)=g(w).$
These functions are bounded, as a consequence of the fact that $T$ is lattice Lipschitz. However, they are not necessarily continuous, even if we define it in a different way when $f(w)=g(w).$ Let us briefly discuss this point with an easy example.
 
\begin{example}
 Consider the Hausdorff compact set $K=[0,1]$ with the usual Euclidean topology. Take two functions $f,g \in C(K)$ such that
 $$
 f-g  = \left\{
	       \begin{array}{ll}
		 1/4 - w     & \mathrm{if\ } 0 \le w < 1/4 \\
		 0  & \mathrm{if \ } 1/4 \le w < 3/4 \\
		 -3/4 +w     & \mathrm{if\ } 3/4  \le w \le 1
	       \end{array}
	     \right.
$$
Let $B=\{f,g\}$ 
 and consider  also an operator $T: B \to C([0,1])$ such that
 $$
 T(f)-T(g)  = \left\{
	       \begin{array}{ll}
		 1/2 - 2 w     & \mathrm{if\ } 0 \le w < 1/4 \\
		 0  & \mathrm{if \ } 1/4 \le w < 3/4 \\
		 -3/4 +w     & \mathrm{if\ } 3/4  \le w \le 1
	       \end{array}
	     \right.
	     .
$$ 
We have that $|T(f)(w)-T(g)(w)| \le 2 |f(w)-g(w)|,$ $w \in [0,1],$ so $T$ is lattice Lipschitz.

Consider the function $\varphi(w):= \frac{|T(f)(w)-T(g)(w)|}{|f(w)-g(w)|},$ and note that it is only defined for $w \in [0,1/4) \cup (3/4,1]$ as
 $$
 \varphi(w)  = \left\{
	       \begin{array}{ll}
		 2   & \mathrm{if\ } 0 \le w < 1/4 \\
		 1     & \mathrm{if\ } 3/4  < w \le 1
	       \end{array}
	     \right.
	     .
$$ 
Therefore, it cannot be completed with  a constant function  in the central interval $ [ 1/4, 3/4 ]$, since each of the values $2, 1$ or $0$ that could be given results in a non-continuous function. Note also that the minimal function $\varphi(w)$  that could be written as functional bound $\varphi$ for the inequality $|T(f)(w)-T(g)(w)| \le \varphi(w) |f(w)-g(w)|,$ $w \in [0,1],$ is
 $$
 \varphi(w)  = \left\{
	       \begin{array}{ll}
		 2   & \mathrm{if\ } 0 \le w < 1/4 \\
		 0  & \mathrm{if \ } 1/4 \le w \le 3/4 \\
		 1     & \mathrm{if\ } 3/4  < w \le 1
	       \end{array}
	     \right.
$$ 
is obviously not continuous, so it does not belong to $C([0,1])$ and does not give an optimal solution to the problem of finding the best continuous functional bound; clearly, this problem has no solution in this case. Therefore, there are lattice Lipschitz operators for which the minimal bound is not a continuous function: define $B=\{f,g\}$ and $T$ as above.
\end{example}

Therefore, we know that the minimal function bound is not in general continuous. However, one could think that, in case that all the functions  $\varphi_{f,g}$ are continuous, we could find a global continuous minimal function bound. One might be tempted to define a continuous function as a uniform minimal  bound for all the  inequalities involving all couples of functions $f,g$ as the lower upper bound of the family $\{ \varphi_{f,g}:f,g \in B \}$ (in the sense of Banach lattices). The problem is that this lower upper bound does not necessarily exist, due to the fact that $C(K)$ is not a complete lattice in general. There is a characterization, due to Nakano and Stone,  of those compact Hausdorff spaces $K$ for which the corresponding space $C(K)$ is $\sigma-$complete or complete, in the lattice sense (see Proposition 1.a.4 in \cite{lint}): a space $C(K)$ is lattice complete if and only if $K$ is extremally disconnected, that is, the closure of every open set in it is again open. For example $\ell^\infty$ is complete, while $C[0,1]$ is not even $\sigma-$complete.  

Summing up the results in this section on the function bounds for lattice Lipschitz operators, we have that

\begin{enumerate}[i)]
	
	\item In general, the minimal function bound for a lattice Lipschitz operator is not a continuous function. But there is always a global continuous function bound, that is provided, for example, by a constant function.
	
	\item Suppose that there is a weakly compact set  $S \subseteq C(K)$ such that for every finite set $W$ of $B$ the associate lattice Lipschitz inequalities holds for a bound given by a certain $\varphi_W \in S.$ Then we can find a global function bound $\varphi \in \overline{co(S)} \subseteq C(K)$ for the lattice Lipschitz operator $T.$
	
\end{enumerate}

\section{Representation of lattice Lipschitz operators on $C(K)$ as vector valued (continuous) functions} \label{S3}

Consider a Banach space $(X,\| \cdot\|)$ and a functional $x' \in X^*$ which defines a seminorm $p_{x'} : X \to \mathbb R$ by $p_{x'}(x) := |\langle x, x' \rangle|,$ $x \in X.$  If $S \subseteq X$ and $T:S \to  X$ is a map, we can write the corresponding Lipschitz-type inequality as
$$
\big| p_{x'}(T(x))  -  p_{x'} (T(y)) \big|   \le  K \big|  p_{x'}(x)- p_{x'}(y)  \big|,  \quad x,y \in S,
$$
that is
$$
\big| \langle T(x)  -  T(y), x' \rangle \big|   \le  K \big|  \langle x- y, x' \rangle  \big|,  \quad x,y \in S.
$$

This kind of maps have been studied as natural generalizations of the notion of Lipschitz maps, although it can be easily seen that the definition is rather restrictive due to the fact that $x'(x)$ equals $0$ for a lot of elements $x \in X,$ and $T(x)$ has to be $0$ too if $y=0$ and $T(0)=0.$ A straightforward argument shows  that the McShane-Whitney extension formulas can be used in this case as well.

Let us focus now our attention on the case $X=C(K).$
Let $S \subseteq C(K)$ and consider a lattice Lipschitz operator $T:S \to C(K).$  Then there is a function $\varphi \in C(K)$ such that for every $w \in K$ and $f,g \in S,$ we have  
$$
|T(f)(w) - T(g)(w) | \le \varphi(w) \, |f(w)- g(w)|.
$$
This can be understood as a sort of uniform Lipschitz seminorm-type map as the ones defined above, for the case of all the seminorms defined by the set of Dirac's deltas $\{\delta_w\}_{w \in K} \subset C(K)^*.$ Write $ C(K)_w$ for the seminormed space $( C(K), |\langle \cdot, \delta_w\rangle|)$ and 
$S_w$ for $(S,  |\langle \cdot, \delta_w\rangle|).$
We can then consider the map  that leads every point $w \in K$ to a Lipschitz map  $P_T(w) \in  Lip(S_w, \mathbb R)$ given by 
$$
P_T(w)(f):=\big( T(f) \big)(w), \quad w \in K, \quad f \in S,
$$
for which the value of the map depends only on the real number $f(w),$ and not on $f$ itself, giving a ``diagonal description" of the operator $T.$
In fact, this representation can be pushed further: next result shows that \textit{every lattice Lipschitz operator is diagonal,} i.e. it can be written as a composition of the evaluation operator 
and the map 
$$
w \mapsto F_T(w) \in Lip(S(w), \mathbb R) \quad \text{ given by} \quad F_T(w)(v)=T(f)(w)
$$
for any
$f \in S$ such that $f(w)=v.$ This way of understanding the definition of the operator aproximates the notion of lattice Lipschitz operator to the papers of Appell et al. \cite{appell,appel2,appel3,appLib}, in which they are considered primarily as superposition operators. 
Note that the subset in the range
$$
S(w)=\big\{ r \in \mathbb R: \, \text{there is a function} \,\, f \in S \,\, \text{such that} \,\, f(w)=r \big\}
\subset \mathbb R
$$
is the image of a real-valued Lipschitz real function, that can be extended to all $\mathbb R.$ This will be needed in the proof of the Proposition \ref{propo1}.

\begin{remark}
Consider again the linear case (explained in Example \ref{exa1} (2)) as basic reference: in this case, a standard diagonal map $C(K) \to C(K)$ is defined as a multiplication operator defined by a continuous function. That is, given $h \in C(K),$ and taking into account that $C(K)$ is an algebra, we have that $L_h:C(K) \to C(K)$ given by $L_h(g)(w)= h(w) \cdot g(w)$ is a linear and continuous map with norm $\| L_h \| = \| h \|_{C(K)}.$  We can understand such operator as a map from $K$ to the set of a linear maps
from $\mathbb R$ to $\mathbb R,$ $K \ni w \mapsto \tau_w \in L(\mathbb R)$ in such a way that
$\tau_w (f(w)):= h(w) \cdot f(w),$ since  the elements of $ L(\mathbb R)$ are just multiplications by a real number, that is in this case just the value of $h(w).$ 

Let us write $\delta_w$ as the evaluation map $C(K) \ni f \mapsto  \delta_w(f)= f(w), $ $w \in K.$ 
Therefore, a multiplication operator can be seen as a composition (in fact, superposition)
$$
L_h(f)(w)= h(w) \cdot f(w)=\tau_w \circ \delta_w(f), \quad w \in K, \quad f \in C(K),
$$
where $\tau_w(r)= h(w) \cdot r,$ $r \in \mathbb R.$ Moreover, $L_\varphi$ can be seen as a function of $C(K, L \big( \mathbb R) \big) \equiv C(K, \mathbb R) = C(K),$ that is, $h.$
\end{remark}

Let us see in the next result that lattice Lipschitz operators play the same role as multiplication operators in the context of the Lipschitz maps.
Recall that $\delta_w$ is the evaluation map $C(K) \ni f \mapsto  \delta_w(f)= f(w), $ $w \in K,$ and for a lattice Lipschitz map $T,$
for every $w \in K$ and every $f,g \in S,$ we have  
$$
|T(f)(w) - T(g)(w) | \le \varphi(w) \, |f(w)- g(w)|,
$$ 
where $\varphi(w)$ is the (best) Lipschitz constant for the pointwise evaluation of the Lipschitz function $T$ at $w.$ Note that, if $\varphi$ is the minimal function satisfying the lattice Lipschitz inequalities, we have that $\sup_{w \in K} \varphi(w) = \| T \|_{\LLip}.$

\begin{theorem}   \label{propo2}
	Let $T:S \to C(K)$ be a lattice Lipschitz operator, $S \subseteq C(K).$ Then, 
	\begin{enumerate}[i)]
		\item
		for every $w \in K,$ the map $S \ni f \mapsto T(f)(w) \in \mathbb R$ factors as
		$$
		T(f)(w) = \phi_w \circ \delta_w(f), \quad f \in S,
		$$
		for a certain function $\phi_w \in \Lip(\mathbb R),$ and the map $K \ni w \mapsto \varphi(w) := \Lip(\phi_w)$ satisfies \eqref{eq:llip_inequality} and belongs to $B(K).$
		That is, for every $w$ there is a factorization as
		$$
		\large
		\xymatrix{
		S \subseteq C(K)	\,\,\, \ar[r]^{ \,\,\,\,\,T(\cdot)(w)}   \ar@{.>}[d]^{\delta_w}&  \,\,\, \mathbb R . \\
				 \mathbb R \ar[ur]_{ \phi_w}\\} 
		$$
		
		\item  $T$ can be identified with a vector valued bounded function  $\Phi_T \in B\big(K,\Lip(\mathbb R) \big)$ and $\varphi(w) = \| \Phi_T(w) \|_{\Lip(\mathbb R)}$.
	\end{enumerate}
\end{theorem}

\begin{proof}
	i) The first thing that needs a proof is that, for every $w \in K,$
	$T(f)(w)$ is independent of the function $f \in S$ whenever its evaluation takes the same value $f(w)=r$ for a certain $r \in S(w).$
	Fix $f \in S$ and $w \in K,$ and suppose that $g \in S$ satisfies that $f(w)=r=g(w).$ Using the lattice Lipschitz inequality for $T,$ we have that
	\begin{equation*}
		\big| T(f)(w)-T(g)(w) \big| \le \varphi(w) \, |f(w)-g(w)| = \phi(w) \cdot 0=0.
	\end{equation*}
	Therefore, $T(f)(w)$ is independent of $f,$ since it only depends on $f(w).$
	Thus, defining $\phi_w(r) = T(f)(w)$ (for any $f \in \C(K)$ such that $f(w) = r$), if $r,s \in S(w) \subset \mathbb R,$ we can select $f, g \in S$ such that $f(w) = r$ and $g(w) = s,$ consequently,
	\begin{align*}
		\big| \phi_w(r)- \phi_w(s)\big| & = \big| T(f)(w)-T(g)(w) \big| \le \varphi(w) \, | f(w)- g(w) | \\
		& = \varphi(w) \, | \delta_w(f)- \delta_w(g) | = \varphi(w) \, |r- s|.
	\end{align*}
	Consequently, $\phi_w$ is well-defined on $S(w)$, and its Lipschitz constant less or equal to $\varphi(w)$.
	In order to get that $\phi_w \in \Lip(\mathbb R),$ we need to see that $\phi_w$ is defined on all the points of $\mathbb R,$ and not only on $S(w).$ But this is a consequence of the McShane-Whitney Theorem, that assures that we can extend it to the whole metric space $\mathbb R$ preserving the Lipschitz constant that is less or equal to $\varphi(w).$ Any such an extension gives the desired factorization.
	  
	ii) Define $\Phi_T$ by $\Phi_T(w):=\phi_w.$ It is a well-defined vector valued function (with values in $\Lip(\mathbb R)$), and norm bounded on $K,$ since the function $w \mapsto \Lip(\phi_w)$ is by definition bounded by the constant $\|\varphi\|_{C(K)},$ where $\varphi$ is a bound function for $T.$
  
 \end{proof}

In Theorem \ref{propo2}, we cannot assert that $\Phi_T$ is a continuous function in general. If it is so, the corresponding lattice Lipschitz operators are easier to represent.

\begin{example}
Let $K = [-1,1].$
	\begin{enumerate}[i)]
		\item  Let $\phi_w(r) = 0$ for $w \leq 0, r \in \R$ and $\phi_w(r) = \max\{0, r - 1/w\}$ for $w > 0, r \in \R$. Consider $\Phi(w) = \phi_w$, obviously, $\Phi \not\in C\big(K, \Lip(\R)\big)$, but the resulting $T: \C(K) \to \C(K)$ given by 
		\begin{equation*}
			T(f)(w) = \phi_w\big(f(w)\big) =
			\left\{ \begin{array}{ll}
				0 & \text{if } w \leq 0 \\
				\max\{ 0, f(w) - \frac{1}{w} \} \  & \text{if } w > 0 \\
			\end{array}\right.
		\end{equation*}
		is a lattice Lipschitz operator.
		
		\item Furthermore, the above conditions does not characterize the set of lattice Lipschitz operators. For example, if $\phi_w = Id_{\R}$ for $w \leq 0$ and $\phi_w = 0$ for $w > 0$, then, the resulting $T$ satisfies the condition \eqref{eq:llip_inequality} but is not well defined from $\C(K)$ to $\C(K)$: $T(1) = \chi_{[-1,0]} \not\in \C(K)$.
	\end{enumerate}
\end{example}

However, in the general case there are still some continuity properties, that in fact characterize the functions representing lattice Lipschitz operators. Let us show them in the next result,
which, in particular, shows that the lattice Lipschitz operators on from subsets of $\C(K)$ to $\C(K)$ are superposition operators that satisfy the Lipschitz condition (see \cite{appell} for related results; see also Ch.6 in \cite{appLib} for the general setting). This is our main representation result.

\begin{theorem} \label{propo1}
	Let $T: \C(K) \to \C(K)$. $T$ is a lattice Lipschitz operator if and only if there exists a function $\psi: K \times \R \to \R$ such that for any $w \in K$
	\begin{equation}
		\label{eq:llip_psi}
		T(f)(w) = \psi \big( w, f(w) \big), \qquad \forall w \in K, \quad f \in \C(S),
	\end{equation}
	satisfying
	\begin{enumerate}[i)]
		\item $\psi(\cdot, r) \in \C(K)$ for any $r \in \R$,
		\item $\psi(w, \cdot) \in Lip(\R)$ for any $w \in K$,
		\item $\Lip \big( \psi(w, \cdot) \big)$ is bounded in $w \in K$.
	\end{enumerate}
	In this case, $\| T \|_{\LLip} = \sup_w \Lip \big( \psi(w, \cdot) \big) = \| \Phi_T \|_{B(K,\Lip(\R))}$, where $\Phi_t(w) = \psi(w,\cdot)$ as in the previous Theorem.
\end{theorem}

\begin{proof}
	The proof of the existence of $\psi$ is a direct consequence of the proof of Theorem \ref{propo2}, letting $\psi(w,r) = \phi_w(r)$.
	Let $w, k \in K$ and $r \in \R$. Choose $f \in \C(S)$ such that $f(w) = f(k) = r$,
	\begin{equation*}
		| \psi(w,r) - \psi(k,r) | = | T(f)(w) - T(f)(k) |,
	\end{equation*}
	and, as a consequence of the continuity of $T(f)$, $\psi(\cdot,r) \in \C(K)$.
	Now, let $w \in K$ and $r, s \in \R$. Then, letting $f, g \in \C(K)$ being any pair of functions such that $f(w) = r$ and $g(w) = s$,
	\begin{equation*}
		| \psi(w,r) - \psi(w,s) | = | T(f)(w) - T(g)(w) | \leq \varphi(w) \cdot | r - s |,
	\end{equation*} 
	so  $ \psi(w, \cdot)$ is Lipschitz and $ \Lip \big( \psi(w, \cdot) \big) \leq \varphi(w) $. This proves \textit{i), ii)} and \textit{iii).} 
	
	For the opposite, assume that $\psi$ satisfies the hypothesis. To show that $T: \C(K) \to \C(K)$ is well defined, fix $f \in \C(K)$ and $w \in K$. Let $k \in K$,
	\begin{align*}
		| T(f)(w) - T(f)(k) |
		& \leq | \psi(w,f(w)) - \psi(k,f(w)) | + | \psi(k,f(w)) - \psi(k,f(k)) | \\
		& \leq | \psi(w,f(w)) - \psi(k,f(w)) | + \Lip \big( \psi(k, \cdot) \big) | f(w) - f(k) |.
	\end{align*}
	The continuity of $\psi(\cdot,f(w))$ and $\sup_k \Lip \big( \psi(k, \cdot) \big) < + \infty$, implies the continuity of $T(f)$ in $w$.
	It remains to show that $T$ is a lattice Lipschitz operator. But this is a direct consequence of the following simple computations. Consider the real function $\varphi: K \to \mathbb R$ given by $\varphi(w) := \sup_k \Lip \big( \psi(k, \cdot) \big),$ $w \in K$. Let $f, g \in \C(K)$ and $w \in K.$ Then
	\begin{align*}
		| T(f)(w) - T(g)(w) |
		& = | \psi(w, f(w)) - \psi(w, g(w)) | \\
		& \leq \Lip \big( \psi(w, \cdot) \big) \cdot | f(w) - g(w) | \\
		& \leq \varphi(w) \cdot | f(w) - g(w) |,
	\end{align*}
what gives the result.

For the equivalence of the norms,
\begin{align*}
	\| T \|_{\LLip} & = \inf \left\{ Q > 0 : \ | T(f) - T(g) | \leq Q | f - g | \ \forall f, g \in \C(K) \right\} \\
	& = \sup_{w \in K} \sup \left\{ \dfrac{ | T(f)(w) - T(g)(w) | }{ | f(w) - g(w) | } : \ f, g \in \C(K), \ f(w) \neq g(w) \right\} \\
	& = \sup_{w \in K} \sup \left\{ \dfrac{ | \psi(w,r) - \psi(w,s) | }{ | r - s | } : \ r, s \in \R, \ r \neq s \right\} \\
	& = \sup_{w \in K} \Lip \psi(w,\cdot) = \sup_{w \in K} \Lip \Phi_T(w).
\end{align*}
\end{proof}


\section{Applications: diagonal Lipschitz operators on $\C(K)-$spaces}   \label{S4}

We have shown in the previous section that the main feature of lattice Lipschitz maps is that they are ``diagonal'', and we have explained in what sense this statement holds.  We will show in this section three applications, which have as a common interest to exploit this idea, by drawing an extended picture of the setting of multiplication operators on $\C(K)-$spaces but for Lipschitz operators.


\subsection{Algebra structure of the lattice Lipschitz maps on $\C(K)$}

As we have discussed, there is a parallelism between the Lipschitz lattice operators (in the Lipschitz environment) and the multiplication operators (in the linear environment) from $C(K)$ to $C(K).$ The set of multiplication operators $M_h :\C(K) \to \C(K)$ defines a unital algebra with the dot product, as a consequence of the algebra structure of $\C(K)$. The same situation appears for the space $\LLip(\C(K)),$ in this case with superposition as internal operation, giving a (noncommutative) unital algebra. This can be easily proved as a consequence of Theorem \ref{propo2}. Recall from this result that for every lattice Lipschitz map $T:\C(K) \to \C(K)$ we find a map $w \mapsto  \phi_w \in Lip(\mathbb R)$ that gives a vector valued function
$\Phi:K \to Lip(\mathbb R)$  such that for every $w \in K,$ $\Phi(w)=\phi_w(f(w).$ This easily gives the next result.

\begin{corollary}
Let $K$ be a compact Hausdorff space and consider the  vector-valued constant function $\Phi_{Id}:K \to Lip(\mathbb R)$ given by   $\Phi_{Id}(w):=Id_{\mathbb R}$ for all $w \in K,$ where $Id_{\mathbb R}$ is the identity function on $\mathbb R.$ Then
\begin{enumerate}[i)]
\item If $T_1$ and $ T_2$ are lattice Lipschitz maps on $C(K)$ with associated bounded vector-valued functions
$\Phi_{T_1}$ and $\Phi_{T_2},$ we have the pointwise inequality
$$
\| \Phi_{T_2 \circ T_1}(w) \|_{Lip(\mathbb R)} \le \| \Phi_{T_2}(w) \|_{Lip(\mathbb R)} \cdot \| \Phi_{T_1}(w) \|_{Lip(\mathbb R)}
$$
for every $w \in K.$

\item The space $LLip_0(C(K))$ with the lattice Lipschitz norm is a (non-commutative) unital Banach algebra under composition, in which the identity element is given by the function $\Phi_{Id}$ described above.
\end{enumerate}
\end{corollary}

\begin{proof}
The space is normed and complete by Proposition \ref{propo:llip_banach}.
Consider a pair of lattice Lipschtz operators $T_1$ and $T_2$, that can be identified with bounded funtions $\Phi_{T_1}$ and $\Phi_{T_2}$ in $B(K,\Lip(\mathbb R)).$
Clearly, $T_2 \circ T_1 : \C(K) \to \C(K)$.
These functions are defined pointwise in $w \in K$ by the Lipschitz functions $\Phi_{T_j}(w):= \phi^j_w:\mathbb R \to \mathbb R,$ $j=1,2.$
Then $\phi^2_w \circ \phi^1_w$ is a Lipschitz function on $\mathbb R.$ A direct calculation shows that $\Lip \big( \phi^2_w \circ \phi^1_w \big) \le \Lip( \phi^2_w ) \cdot \Lip( \phi^1_w),$  and for every $w \in K,$ 
\begin{equation*}
	(T_2 \circ T_1) (f)(w)= \phi^2_w \big(  \phi^1_w( f(w)) \big)= \big(\phi^2_w \circ \phi^1_w \big)( f(w)) =
	\Phi_{T_2 \circ T_1}(w) (f(w)).
\end{equation*}
This shows that $T_2 \circ T_1$ is a lattice Lipschitz operator. The same direct computations that show the  inequalities for the Lipschitz norm of  the  pointwise defined composition 
$\phi^2_w \circ \phi^1_w$ give also
$$
\| T_2 \circ T_1 \|_{\LLip(C(K))}  \leq \| T_2  \|_{\LLip(C(K))} \cdot \| T_1 \|_{\LLip(C(K))}.
$$
Finally, it is easy to see that the identity map is represented by the vector valued function  $\Phi_{Id}(\cdot)=Id_{\mathbb R},$ and it defines a unit for the space of lattice Lipschitz operators.  
\end{proof}


\subsection{Dual spaces of spaces of lattice Lipschitz operators} 

Let us show a first application of the representation results for lattice Lipschitz operators on $C(K)-$spaces as vector-valued functions. 
We have shown in Section \ref{S3} that we can always find a bound function $\varphi$ that is continuous, although the minimal one need not be so. In this sense, we can consider the subspace $L_c Lip(C(K))$ of lattice Lipschitz maps $T$ having such that the representing function
$\Phi_T$  provided by Theorem \ref{propo2}  belongs to $\C(K,Lip(\mathbb R)),$ and hence with a  continuous minimal bound function. 
Using classical representations of spaces vector-valued functions, we can also give a description of the dual spaces of certain subspaces of $LLip(C(K)).$ The equivalent results for the case of linear maps are quite simple, but provide a good benchmark: in the linear case, diagonal maps are functions of $C(K)$ and their norms are exactly the $C(K)-$norm of the associated continuous function. So, by Riesz's Theorem, its dual space is just the space $\mathcal M(K)$ of Borel regular measures on $K,$ and the dual action of a linear diagonal operator
$T_g$ with such a measure $\mu$ is the integral
$$
\big\langle T_g, \mu \big\rangle = \int_K g \, \, d \mu.
$$
Let us show that, in this case, the dual space can also be described by a (rather more complicated) integral, preserving in a sense the same spirit, although we are now talking about the (much) larger class of diagonal Lipschitz maps.

So, let us center our attention on the subspace $L_c Lip(C(K))\hookrightarrow LLip(C(K))$ endowed with the same norm,and its  identification  with  $C \big(K,Lip(\mathbb R \big) \big)$ provided \textit{via} Theorem \ref{propo2} and Theorem \ref{propo1}.  A space of Banach-space-valued continuous on a compact set can always be written as the completion of an injective tensor product (see for example \cite[4.2(2)]{deflo}, and so we have that
$$
L_c Lip(C(K)) = C \big(K,Lip(\mathbb R \big) \big)= C(K) \widehat{\otimes}_\varepsilon Lip(\mathbb R).
$$
Therefore, each such a lattice Lipschitz operator can be approximated with respect to the $\varepsilon$ tensor norm by a sequence of simple tensors $t=\sum_{i=1}^n f_i \otimes \phi_i,$
that acts as a Lipschitz operator $T_t:C(K) \to C(K)$ by the formula
$$
T_t(f)(w)= \sum_{i=1}^n f_i(w) \cdot \phi_i(f(w)), \quad f \in C(K), \,\,\, \, w \in K.
$$

In particular, this representation  provides a description of the dual space. Indeed, recall that every element of the dual space of an injective tensor product as
$C(K) \otimes_\varepsilon \Lip(\mathbb R) $ can be written as an \textit{integral bilinear form} \cite[Theorem 4.6]{deflo}. This gives the following.

\begin{remark}
	Consider the identification 
$		\Lambda : \C(K) \otimes \Lip(\R)  \to \LLip(\C(K)) $ given by
$$
		u = \sum_{i=1}^n f_i \otimes \phi_i  \mapsto \Lambda(u) : f \mapsto \left( w \mapsto \sum_{i=1}^n f_i(w) \cdot \phi_i(f(w)) \right).
$$
	Then, $\Lambda(u)(f)(w) = \psi_u(w,f(w))$, where $\psi_u(w,r) = \sum_{i=1}^n f_i(w) \cdot \phi_i(r)$, so by Theorem \ref{propo1}, $\Lambda$ is well defined.
	Clearly, it is a linear mapping. Let us show that it preserves the injective tensor norm in $\C(K) \otimes_\varepsilon \Lip(\R)$. Indeed, we have that
	\begin{align*}
		\| \Lambda(u) \|_{\LLip} & = \sup_{w \in K} \Lip \big( \psi(w, \cdot) \big)
		= \sup_{w \in K} \Lip \left( \sum_{i=1}^n f_i(w) \cdot \phi_i(\cdot) \right) \\
		& = \| u \|_{C(K,\Lip(\R))} = \varepsilon(u).
	\end{align*}

	Then, $\Lambda$ can be extended by continuity to the completion of the injective tensor product,
	\begin{equation*}
		\Lambda : \C(K) \hat\otimes_\varepsilon \Lip(\R) \to \LLip(\C(K)),
	\end{equation*}
	which allows to identify isometrically $\C(K) \widehat\otimes_\varepsilon \Lip(\R)$ with $\LcLip(\C(K))$, the image of $\Lambda$, as a subset of $\LLip(\C(K))$.
%
\end{remark}


\begin{corollary}
Every element  $\psi \in  \big( C(K) \otimes_\varepsilon Lip(\mathbb R) \big)^* = L_c Lip(C(K))^*$ can be written as a bilinear integral, i.e. there is  a Borel-Radon measure $\eta$  on $B_{\mathcal M(K)} \times B_{Lip(\mathbb R)^*}$ such that for every $t \in C(K) \otimes_\varepsilon Lip(\mathbb R),$
$$
\langle  t, \psi \rangle = \int_{B_{\mathcal M(K)} \times B_{Lip(\mathbb R)^*}} \big\langle \, t \, , \,  \mu \otimes \tau \big\rangle \,\,\, d \eta(\mu,\tau).
$$
where $ \big\langle \, t \, , \,  \mu \otimes \tau \big\rangle =  \int_K f \, d \mu \cdot \langle \tau,\phi\rangle$
for simple tensors as $t= f \otimes \phi \in C(K) \otimes Lip(\mathbb R).$
\end{corollary}

Using well-known notions of tensor product representations of operator ideals (\cite[Proposition 10.1]{deflo}), we can write the same result in other words as

\begin{corollary}
$L_c Lip(C(K))^* = \mathcal I \big( C(K), Lip(\mathbb R)^* \big)$ isometrically, where $\mathcal I$ denotes the ideal of the integral operators.
\end{corollary}

Recall that the representation has to be done taking into account the identification of operators with functions. That is,  every element $T$ of $L_c Lip(C(K))$ is considered  as a continuous vector valued function $w \mapsto \Phi_T(w) \in L(\mathbb R),$ which acts on a continuous function $f \in C(K)$ as $T(f)(w)=\Phi_T(w)(f(w)),$ $w \in K,$ and $\Phi_T$ can be approximated in the norm $\varepsilon$ by simple tensors  like the tensor $t$ written above.

 
\subsection{Extension of lattice Lipschitz maps}  

Let us finish the paper with an application regarding a vector valued version of the classical  McShane-Whitey extension theorem for real valued Lipschitz maps.
In this section we show that it is possible to obtain a well-defined extension $\tilde{T}:C(K) \to C(K)$ of a lattice Lipschitz  operator $T:S \to C(K)$ provided that we use a continuous function as  bound function $\varphi$ for $T.$ We focus attention on this result because of the original motivation of this paper, which was to find good extension theorems for this class of maps in order to apply them in the context of Reinforcement Learning in Artificial Intelligence. In this section we show that this is always possible.

Let $K$ be a compact metric space and $S \subseteq \C(K)$. Recall that we say that $T : S \to \C(K)$ is a lattice Lipschitz operator if there exists $\varphi \in \C(K)$ such that for each $f, g \in S$, $| T(f) - T(g) | \leq \varphi \cdot | f - g |$.
However, in general we cannot assure that such a continuous $\varphi$ is the minimal function satisfying the inequality, as explained in Section \ref{S2}.  

Let us write $\mathcal F(K)$ for the space of all the real functions on $K.$
The (pointwise) McShane extension of $T$ to $C(K)$ is defined as the function $\hat T: \C(K) \to \mathcal F(K)$ given by
\begin{equation*}
	\hat T(f)(w) = \sup \{ T(g)(w) - \varphi(w) | g(w) - f(w) | : \, g \in S \}, \quad w \in K.
\end{equation*}
Note that this formula does not necessarily define, a priori,  a function in $\C(K)$ for every continuous $f.$
But next result shows that it provides a well-defined lattice Lipschitz operator from $\C(K)$ to $\C(K)$ with the expected properties. 

\begin{theorem}  \label{extcont}
	Let $S \subseteq \C(K)$ and let  $T : S \to \C(K)$ be a lattice Lipschitz operator with associate function $\varphi \in \C(K)$. Then, the pointwise McShane extension $\hat T : \C(K) \to \C(K)$ is a well-defined lattice Lipschitz extension of $T$ to $\C(K)$ with the same associate function, $\varphi$.
\end{theorem}

\begin{proof} The proof follows the lines of the classical one for real functions; we write it for the aim of completeness.
	Let $f \in S$ and $w \in K$. Then, clearly $\hat T(f)(w) \geq T(f)(w)$ and
	\begin{equation*}
		\hat T(f)(w) - T(f)(w) = \sup \{ T(g)(w) - \varphi(w) \cdot | g(w) - f(w) | - T(f)(w) : \, g \in S \} \leq 0,
	\end{equation*}
	so $\hat T$ extends $T$. To show that the lattice Lipschitz associate function is preserved, let $f, g \in \C(K)$ and $w \in K.$ Then
	\begin{align*}
		| T(f)(w) - T(g)(w) |
		& \leq \Big| \sup_{h \in S} \{ T(h)(w) - \varphi(w) \cdot | h(w) - f(w) | \\
		& \hspace{2cm} - T(h)(w) + \varphi(w) \cdot  | h(w) - g(w) | \} \Big| \\
		& \leq \varphi(w) \cdot | f(w) - g(w) |.
	\end{align*}
	
	It remains to prove that $\hat T(f) \in \C(K)$ for all $f \in \C(K)$. Let $f \in \C(K)$ and $\varepsilon > 0$. Fix $w \in K.$ Then, for every $k \in K$,
	\begin{align*}
		| \hat T(f)(w) - \hat T(f)(k) |
		& \leq \Big| \sup_{g \in S} \{ T(g)(w) - \varphi(w) \cdot | g(w) - f(w) | \\
		& \hspace{2cm} - T(g)(k) + \varphi(k) \cdot | g(k) - f(k) | \} \Big| \\
		& \leq \Big| \sup_{g \in S} \{ | T(g)(w) - T(g)(k) | \\
		& \hspace{2cm} + | \varphi(w) g(w) - \varphi(w) f(w) - \varphi(k) g(k) + \varphi(k) f(k) | \} \Big| \\
		& \leq \sup_{g \in S} \{ | T(g)(w) - T(g)(k) | + \\
		& \hspace{2cm} + | \varphi(w) g(w) - \varphi(k) g(k) | + | \varphi(w) f(w) - \varphi(k) f(k) | \}.
	\end{align*}
	Now, note that there exists $g_0 \in S$ such that
	\begin{align*}
		| \hat T(f)(w) - & \hat T(f)(k) | \leq | T(g_0)(w) - T(g_0)(k) | \\
		& + | \varphi(w) g_0(w) - \varphi(k) g_0(k) | + | \varphi(w) f(w) - \varphi(k) f(k) | + \varepsilon
	\end{align*}
	and, by the continuity of the functions $T(g_0), \varphi \cdot g_0$ and $\varphi \cdot f$, there is a  neighborhood $\mathcal N \subset K$ of $w$
for which $k \in \mathcal N$ implies $ | T(g_0)(w) - T(g_0)(k) | \leq \varepsilon, | \varphi(w) g_0(w) - \varphi(k) g_0(k) | \leq \varepsilon,$ and $ | \varphi(w) f(w) - \varphi(k) f(k) | \leq \varepsilon $.
	As a consequence, if $k  \in \mathcal N,$
	\begin{equation*}
		| \hat T(f)(w) - \hat T(f)(k) | \leq 4 \varepsilon.
	\end{equation*}
\end{proof}


The requirement on the continuity of the function $\varphi$ appearing in the previous result is necessary. Indeed, 
 observe that if $\varphi \in B(K) \setminus \C(K)$, the extension formula may not work. Let us show an example of such situation.

\begin{example}
	Take $K = [-1,1]$ and $f = 0, g = Id_K : K \to \R$ .
	Consider $S = \{f, g\}$ and $T : S \to \C$ defined as $T(f) = 0$ and $T(g) = Id_K \cdot \chi_{[0,1]}$.
	Then, $T$ is a lattice Lipschitz operator with (optimal) bound function $\varphi = \chi_{[0,1]} \in B(K) \setminus \C(K)$.
	But the McShane extension of $T$ does not map $\C(K)$ into $\C(K)$.
	To see this, let $h$ be the constant function $1$ and note that
	\begin{equation*}
		\hat T(h)(w) = \left\{
		\begin{array}{ll}
			0 & \text{if} \ w < 0 \\
			- 1 + 2w \ & \text{if} \ w \geq 0,
		\end{array} \right.
	\end{equation*}
	a non-continuous function.
	Observe that we get a different result if we take another bound function $\hat\varphi \in \C(K)$ as the one given by  $\hat\varphi(w) = 1 + w,$ if $w < 0,$ and $\hat\varphi(w) = 1,$ if $w \geq 0$. By Theorem \ref{extcont}, the McShane extension will be a (well defined) lattice Lipschitz extension of $T$, in particular,
	\begin{equation*}
		{\hat T}(h)(w) = \left\{
		\begin{array}{ll}
			- 1 - w & \text{if} \ w < 0 \\
			- 1 + 2w \ & \text{if} \ w \geq 0
		\end{array} \right\}
	\end{equation*}
which of course is a continuous function.
\end{example}

\end{document}